\documentclass[12pt]{amsart}
\usepackage{amssymb}
\usepackage{amsthm}
\usepackage{amsfonts}
\usepackage{mathtools}
\usepackage{amsmath}
\usepackage{enumitem}
\usepackage{mathrsfs}
\usepackage{hyperref}
\usepackage{tikz-cd}
\usepackage{MnSymbol}
\usepackage{inputenc}
\usepackage{times}
\usepackage{graphicx}
\usepackage{subfigure}
\usepackage{color}
\usepackage{comment}
\usepackage{stmaryrd}

\usepackage{slashed}

\definecolor{blue2}{cmyk}{.94,.11,0,0}
\theoremstyle{plain}
\newtheorem{mythm}{Theorem}[section]

\newtheorem{mylem}[mythm]{Lemma}
\newtheorem{mypro}[mythm]{Proposition}
\newtheorem{mycor}[mythm]{Corollary}

\newtheorem{myrmk}[mythm]{Remark}

\newtheorem{mydef}[mythm]{Definition}

\newtheorem{myassumpt}[mythm]{Assumption}

\numberwithin{equation}{section}
\DeclareMathOperator{\Id}{Id}

\DeclareMathOperator{\re}{Re}
\DeclareMathOperator{\im}{Im}

\DeclareMathOperator{\Vol}{Vol}

\DeclareMathOperator{\ind}{ind}

\DeclareMathOperator{\Cl}{Cl}

\newcommand{\inumber}{\sqrt{-1}}

\newcommand{\Z}{\bb{Z}}

\newcommand{\R}{\bb{R}}

\newcommand{\w}{\wedge}
\newcommand{\tb}{\text{ }}

\newcommand{\sqrthalf}{\frac{1}{\sqrt{2}}}

\newcommand{\J}{\overline{\delta}}

\newcommand\bb[1]{\mathbb{#1}}

\newcommand{\derivativef}{\alpha}
\newcommand{\hyperconn}{\nabla^{\Sigma}}
\newcommand{\hypermodconn}{\Tilde{\nabla}^{\Sigma}}

\newcommand{\hyperc}{c^{\Sigma}}

\newcommand{\hyperdirac}{\slashed{D}^{\Sigma}}
\newcommand{\hypermoddirac}{\Tilde{\slashed{D}}^{\Sigma}}

\newcommand{\xione}{\xi^{\Sigma}}
\newcommand{\xitwo}{T\Sigma_{\text{cpx}}}
\newcommand{\TSicpx}{\xitwo}
\newcommand{\TSileg}{T\Sigma_{\text{lag}}}
\newcommand{\NSicpx}{N\Sigma_{\text{cpx}}}
\newcommand{\NSileg}{N\Sigma_{\text{lag}}}
\newcommand{\NSi}{N\Sigma}
\newcommand{\TSi}{T\Sigma}
\newcommand{\TdSicpx}{T^{*}\Sigma_{\text{cpx}}}
\newcommand{\TdSileg}{T^{*}\Sigma_{\text{lag}}}
\newcommand{\NdSicpx}{N^{*}\Sigma_{\text{cpx}}}
\newcommand{\NdSileg}{N^{*}\Sigma_{\text{lag}}}
\newcommand{\NdSi}{N^{*}\Sigma}
\newcommand{\TdSi}{T^{*}\Sigma}

\newcommand{\tTSicpx}{T\Tilde{\Sigma}_{\text{cpx}}}
\newcommand{\tTSileg}{T\Tilde{\Sigma}_{\text{lag}}}
\newcommand{\tNSicpx}{N\Tilde{\Sigma}_{\text{cpx}}}
\newcommand{\tNSileg}{N\Tilde{\Sigma}_{\text{lag}}}

\newcommand{\tTSi}{T\Tilde{\Sigma}}

\newcommand{\ic}{i_{c}}
\newcommand{\il}{i_{l}}
\newcommand{\ipc}{i'_{c}}
\newcommand{\ipl}{i'_{l}}
\newcommand{\ip}{i'}
\newcommand{\bari}{\bar{i}}
\newcommand{\barI}{\bar{I}}
\newcommand{\jc}{j_{c}}
\newcommand{\jl}{j_{l}}
\newcommand{\jpc}{j'_{c}}
\newcommand{\jpl}{j'_{l}}
\newcommand{\jp}{j'}
\newcommand{\barj}{\bar{j}}
\newcommand{\barJ}{\bar{J}}
\newcommand{\kc}{k_{c}}
\newcommand{\kl}{k_{l}}
\newcommand{\kpc}{k'_{c}}
\newcommand{\kpl}{k'_{l}}

\newcommand{\bark}{\bar{k}}
\newcommand{\barK}{\bar{K}}

\def\QTR#1#2{{\csname#1\endcsname #2}}

\begin{document}

\title{A Reilly-type inequality for CR Dirac operators on CR manifolds and
applications}

\author{Jih-Hsin Cheng}
\address{Institute of Mathematics, Academia Sinica, Taipei, Taiwan, ROC}
\email{jhcheng@alum.sinica.edu.tw}

\author{Hung-Lin Chiu}
\address{Department of Mathematics, National Tsing-Hua University, Hsinchu, Taiwan, ROC}
\email{hlchiu@math.nthu.edu.tw}

\author{Wei-Ting Kao}
\address{Institute of Mathematics, Academia Sinica, Taipei, Taiwan, ROC}
\email{wtkao@as.edu.tw}


\dedicatory{Dedicated to Professor Sorin Dragomir on the occasion of his
70th birthday}



\begin{abstract}
The first two authors derived the
Weitzenb\"{o}ck-type Formula for a CR Dirac operator. In this paper, we
apply the above-mentioned formula and follow the process as Riemannian case to derive a Reilly-type inequality. We give two
applications for the CR Reilly inequality. One is a lower bound estimate of
the p-mass in terms of the first eigenvalue of the modified hypersurface CR
Dirac operator. Another one is to solve the CR Dirac equation on a bounded
domain with an $S^{1}$-action and an APS-type boundary condition.
\end{abstract}

\maketitle

\section{Introduction}

We learn from the Riemannian geometry that the Reilly formula/inequality for
a domain with boundary is useful for the study of the spectral analysis, in
particular, some eigenvalue estimates (see, for instance, \cite%
{farinelli1998spectrum}, \cite{hijazi2001dirac}). In this paper, we explore
a Reilly-type formula/inequality for the CR Dirac operator on a (smooth)
domain $\Omega $ with boundary $\Sigma $ in a CR/pseudohermitian manifold.
For basic material in CR/pseudohermitian geometry, we refer the reader to
Sections \ref{Sec2} and \ref{Sec3}. For operators $\Tilde{\slashed{D}}^{\Sigma
},$ $\slashed{D},$ $\Lambda ,$ $\nabla ,$ see (\ref{mhDirac}), (\ref{Dirac}), (%
\ref{Cl2}), (\ref{spinconndef}) (and Remark \ref{R-2-z}), respectively. Let $%
H,$ $W$ denote the $p$-mean curvature (see (\ref{pmc})) and the
Tanaka-Webster scalar curvature (see, for instance, (\ref{2-4a}) in \cite%
{chengchiu2022positivemass5}) respectively. For $T$ the Reeb vector field,
see (\ref{Rvf}). For $d\Sigma $ and $dV$, see (\ref{dsigma}) and (\ref{dv})
respectively.

\begin{mythm}
\label{thm1} With the notations above, we have 1) a Reilly-type formula for
the CR Dirac operator $\slashed{D}$: 
\begin{equation}  \label{1-1a}
\begin{aligned} &\oint_{\Sigma} \langle \Tilde{\slashed{D}}^{\Sigma} \psi,
\psi \rangle \, d\Sigma - \frac{1}{2}\oint_{\Sigma} H |\psi|^{2} \, d\Sigma
\\ &= \int_{\Omega} \Big[ -|\slashed{D} \psi|^{2} + W |\psi|^{2} + |\nabla
\psi|^{2} - 2\re[ \big\langle \Lambda \nabla_{T} \psi, \psi \big\rangle]
\Big] dV. \end{aligned}
\end{equation}%
\newline

2) a Reilly-type inequality for the CR Dirac operator $\slashed{D}$: 
\begin{equation}  \label{1-2a}
\begin{aligned} &\oint_{\Sigma} \langle \Tilde{\slashed{D}}^{\Sigma} \psi,
\psi \rangle \, d\Sigma -\frac{1}{2} \oint_{\Sigma} H |\psi|^{2} \, d\Sigma
\\ &\geq \int_{\Omega} \Big[ -\tfrac{2n-1}{2n} |\slashed{D} \psi|^{2} + W
|\psi|^{2} - 2\re[ \big\langle \Lambda \nabla_{T} \psi, \psi \big\rangle]
\Big] dV. \end{aligned}
\end{equation}
\end{mythm}

For the singularity issue on the boundary $\Sigma$, we refer the reader to
Remark \ref{R-4-1}.

Next, we want to apply the above formula \ref{1-1a} to an asymptotically
flat pseudohermitian manifold. For such a pseudohermitian manifold $%
(M,J,\theta ),$ we can talk about the $p$-mass $m(J,\theta )$ (see \cite%
{chengchiu2022positivemass5}; the $p$-mass plays a key role in solving the
Yamabe minimizer problem on CR manifolds). As an application of the above
Reilly-type formula, we obtain an integral formula for $m(J,\theta )$ below.

\begin{mythm}
\label{pmassthm} Let $(M,J,\theta )$ be an asymptotically flat
pseudohermitian manifold with an inner bounded boundary $\Sigma $. Assume $%
\dim M=5$ and $\psi \in \Gamma (M,\slashed{S}^{odd})$ (see (\ref{2-5a}) for the
notation $\slashed{S}^{odd}$) is a solution of $\slashed{D}^{2}\psi =0$ with some
boundary condition in \cite{chengchiu2022positivemass5} (see Proposition \ref%
{P-5-1}). Then we have 
\begin{equation}
\begin{aligned} &cm(J,\theta) -\oint\limits_{\Sigma} \langle
\Tilde{\slashed{D}}^{\Sigma} \psi, \psi \rangle
d\Sigma+\oint\limits_{\Sigma} \frac{1}{2} H|\psi|^{2} d\Sigma\\
=&\int\limits_{\Omega} [-|\slashed{D} \psi|^{2}+W|\psi|^{2}+|\nabla
\psi|^{2}]dV \end{aligned}  \label{pmassformula}
\end{equation}%
where $c$ is a positive constant.
\end{mythm}

Denote by $\slashed{S}^{\Sigma ,odd}$ the odd spinor bundle $\slashed{S}^{odd}$
restricted to $\Sigma $ (cf. (\ref{2-5a}), (\ref{3-9a})).

\begin{mycor}
\label{pmasscor} Suppose we are in the situation of Theorem \ref{pmassthm}
with the Tanaka-Webster (scalar) curvature $W\geq 0$. Assume further that
the solution $\psi $ in \cite{chengchiu2022positivemass5} (cf. Proposition %
\ref{P-5-1}) satisfies $\slashed{D}\psi =0$ and $\Sigma $ is a $p$-minimal
hypersurface in $M$. Then we have 
\begin{equation}
cm(J,\theta )\geq \lambda _{1}||\psi ||_{L^{2}(\Sigma ,\slashed{S}^{\Sigma
,odd})}^{2}  \label{massineqn}
\end{equation}%
where $\lambda _{1}$ is the first positive eigenvalue of $\Tilde{\slashed{D}}%
^{\Sigma}$ in $L^{2}(\Sigma ,\slashed{S}^{\Sigma ,odd})$.
\end{mycor}

We consider a general inhomogeneous boundary value problem for the CR Dirac
operator (see the notation $\pi _{+}$ below in Definition \ref{D-6-z}): 
\begin{equation}
\begin{array}{rlc}
\slashed{D}\psi & =\phi & \text{on }\Omega , \\ 
\pi _{+}(\psi) & =\pi _{+}(\rho ) & \text{on }\Sigma .%
\end{array}
\label{1-5}
\end{equation}

\noindent The notation $\llangle\phi ,\chi \rrangle$ in the theorem below
means the integral of $\langle \phi ,\chi \rangle $ over $\Omega .$ For $%
N_{0}(\slashed{D})$, we refer to Definition \ref{D-6-1}. Let $H^{1}_{FS}(\Omega
,\slashed{S})$ denote the first-order Folland-Stein space. Denote by $L^{2}F$ the $L^2$-completion of a function space $F$. 

\begin{mythm}
\label{T-1-4} Let $\Omega $ be a bounded domain with boundary $\Sigma $ in a
s.p.c pseudohermitian manifold. Suppose that there exists an $S^{1}$ CR
action on $\overline{\Omega }$ satisfying Assumption \ref{assum}. Then for
each $\phi \in L_{\leq N}^{2}(\Omega ,\slashed{S})$ and $\rho \in H^{1}_{FS}(\Omega
,\slashed{S})\cap L^{2}\Gamma _{\leq N}(\overline{\Omega },\slashed{S})$ (see (\ref{6-7}) for $%
\Gamma _{\leq N}(\overline{\Omega },\slashed{S})$) satisfying the integrability condition 
\begin{equation}
\llangle\phi ,\chi \rrangle+\oint\limits_{\Sigma }\langle c(e^{2n})\rho
,\chi \rangle =0\text{ }\text{ }\forall \chi \in N_{0}(\slashed{D}),
\end{equation}%
the boundary value problem (\ref{1-5}) has a solution, which is unique up to 
$\Tilde{\chi}\in N_{0}(\slashed{D})$.
\end{mythm}

The paper is organized as follows. In Section \ref{Sec2}, we provide some
basic material in pseudohermitian geometry and spin geometry on the contact
bundle. In Section \ref{Sec3}, we give a spinorial theory on hypersurfaces
in a CR manifold. Among others, we define a modified hypersurface Dirac
operator $\Tilde{\slashed{D}}^{\Sigma }$ and show that $\Tilde{\slashed{D}}%
^{\Sigma }$ is self-adjoint. We then make use of this operator to establish
a Reilly-type formula/inequality for a domain with boundary in a CR manifold
in Section \ref{Sec4}; Theorem \ref{thm1} then follows. In Section \ref{Sec5}%
, we apply the Reilly-type formula in Theorem \ref{thm1} to an
asymptotically flat pseudohermitian manifold with a bounded boundary $\Sigma 
$ to obtain a p-mass formula in Theorem \ref{pmassthm}. We then have (\ref%
{massineqn}) in Corollary \ref{pmasscor}: that is, the lower bound estimate
of the p-mass by the first positive eigenvalue of $\Tilde{\slashed{D}}^{\Sigma
} $ under the condition that $W\geq 0,$ $\Sigma $ is $p$-minimal and the
spinor $\psi $ satisfies the Dirac equation $\slashed{D}\psi =0.$ In Section %
\ref{Sec6}, we solve the CR Dirac equation for a domain with boundary under
an $S^{1}$-action assumption.

\bigskip

\textbf{Acknowledgement. }J.-H. Cheng (resp. H.-L. Chiu) would like to thank
the National Science and Technology Council of Taiwan for the support: grant
no. 112-2115-M-001-012 (resp. grant no. 112-2115-M-007-009-MY3).

\bigskip

\section{Preliminaries and Clifford Actions\label{Sec2}}

Let $(M^{2n+1},J,\theta )$ be a strongly pseudoconvex (spc) pseudohermitian
manifold of dimension $2n+1$ (\cite{Web78}, \cite{Lee86}, \cite{Dra06}). Let 
$\xi =\ker \theta $ denote the contact distribution of $M$. Since $M$ is
spc, we have the Levi metric 
\begin{equation*}
\langle U,V\rangle :=L_{\theta }(U,V)=\frac{1}{2}d\theta (U,JV)\text{ for }%
U,V\in \xi ,
\end{equation*}%
and the volume form 
\begin{equation}
dV:=d\Vol_{\theta }=c_{n}\theta \wedge (d\theta )^{n}.  \label{dv}
\end{equation}%
where $c_{n}=\frac{1}{2^{n}}$. Let $T^{1,0}M$ be the $i$-eigenspace of $J$
on $\mathbb{C}\xi $, and $T^{0,1}M=\overline{T^{1,0}M}$. By \cite{Web78}, we
have the unique Tanaka-Webster connection on $M$. It is not difficult to
write down the real version of Tanaka-Webster structure equations (The
details can be seen in the Appendix in \cite{kao2025thesis}) as follows: 
\begin{equation}
\left\{ 
\begin{array}{rll}
d\theta & =\overline{\delta }_{ij}\,e^{i}\wedge e^{j}, &  \\ 
de^{i} & =e^{j}\wedge \omega _{j}{}^{i}+A^{i}{}_{j}\,\theta \wedge e^{j}, & 
\\ 
0 & =\omega _{i}{}^{j}+\omega _{j}{}^{i}, &  \\ 
0 & =A_{ij}-A_{ji}, &  \\ 
-\omega _{i}^{\text{ }\text{ }j} & =\bar{\delta}_{i}{}^{k}\omega _{k}{}^{l}%
\bar{\delta}_{l}{}^{j}, &  \\ 
A^{i}{}_{j} & =\bar{\delta}_{k}{}^{i}A^{k}{}_{l}\bar{\delta}_{j}{}^{l}. & 
\end{array}%
\right.  \label{realstreqn}
\end{equation}%
where $\xi =\langle e_{1},e_{2},...,e_{2n}\rangle $ with $Je_{\beta
}=e_{\beta +n}$ for $\beta =1,...,n$, the coframe $\{e^{i}\}$ is dual to $%
\{e_{i}\}$, the Tanaka-Webster connection satisfies $\nabla e_{i}=\omega
_{i}^{\text{ }\text{ }j}e_{j}$, the torsion is given by $\langle {\csname up%
\endcsname Tor}(T,e_{i}),e_{j}\rangle =:A_{ij},$ where $T$ denotes the
(Reeb) vector field such that 
\begin{equation}
\theta (T)=1,d\theta (T,\cdot )=0,  \label{Rvf}
\end{equation}%
and $\overline{\delta }_{ij}:=-\langle e_{i},Je_{j}\rangle $. In this
viewpoint, the frame $\{e_{i}\}$ is orthonormal with respect to the Levi
metric. Thus, 
\begin{equation}
\lbrack e_{i},e_{j}]=-2\overline{\delta }_{ij}T+(\omega _{j}^{\text{ }\text{ 
}k}(e_{i})-\omega _{i}^{\text{ }\text{ }k}(e_{j}))e_{k}.  \label{eiej}
\end{equation}%
Here the indices $\beta ,\gamma ,...$ run from $1$ to $n$, and the indices $%
i,j,k,...$ run from $1$ to $2n$.


\subsection{Spin structures on $\protect\xi $}

We refer the material in this subsection to \cite{chengchiu2022positivemass5}%
.

\begin{mylem}
Suppose $(M, \xi)$ is an orientable contact manifold. Then the contact
bundle $\xi$ is spin if and only if the tangent bundle $TM$ is spin.
\end{mylem}

\begin{proof}
See Lemma 2.1 in \cite{chengchiu2022positivemass5}.
\end{proof}

\subsection{Spinor bundles}

Take a type (1,0) coframe $\{\theta ^{1},...,\theta ^{n}\}$ in $T^{\ast
1,0}M $, where 
\begin{equation}
\theta ^{\beta }=e^{\beta }+\sqrt{-1}e^{\beta +n},  \label{cpxtheta}
\end{equation}%
and its dual frame 
\begin{equation}
Z_{\beta }=\frac{1}{2}(e_{\beta }-\sqrt{-1}e_{\beta +n}).  \label{cpxe}
\end{equation}%
Let the conjugate $\{\theta ^{\bar{1}},...,\theta ^{\bar{n}}\}$ be the type
(0,1) coframe in $T^{\ast 0,1}M$. The spinor bundle is defined as 
\begin{equation*}
\slashed{S}=\Lambda ^{0,\ast }T^{\ast 0,1}M=\langle f\theta ^{\overline{i_{1}}%
}\wedge ...\wedge \theta ^{\overline{i_{k}}}\mid i_{j}\in
\{1,...,n\},\,k=0,..,n,\,f\in C^{\infty }(M)\rangle .
\end{equation*}%
The complex Clifford action $c:\Cl(\mathbb{C}T^{\!\ast }M)\;\longrightarrow
\;\mathrm{End}(\slashed{S})$ is defined by 
\begin{equation*}
c(\theta ^{\overline{\beta }})\;=\;\sqrt{2}\,\varepsilon _{\beta },\qquad
c(\theta ^{\beta })\;=\;-\sqrt{2}\,\iota _{\beta },
\end{equation*}%
where $\varepsilon _{\beta }\psi =\theta ^{\overline{\beta }}\wedge \psi $
and $\iota _{\beta }\psi =\iota _{Z_{\overline{\beta }}}\psi $.

By (\ref{cpxtheta}), the corresponding real Clifford action $c: \Cl%
(T^{\!*}M) \;\longrightarrow\; \mathrm{End}(\slashed{S})$ is given by 
\begin{equation*}
c(e^{\beta}) \;=\; \varepsilon_{\beta} - \iota_{\beta}, \qquad
c(e^{\beta+n}) \;=\; \sqrt{-1}\big(\varepsilon_{\beta} + \iota_{\beta}\big).
\end{equation*}

We can also define the even and odd spinor bundles by 
\begin{equation}
\slashed{S}^{even}:=\Lambda ^{0,even}T^{\ast 0,1}M,\quad
\slashed{S}^{odd}:=\Lambda ^{0,odd}T^{\ast 0,1}M.  \label{2-5a}
\end{equation}

\begin{mydef}
\label{spinconn} The spinor connection on the spinor bundle is defined by 
\begin{equation}
\nabla _{X}\psi =X\psi +\frac{1}{4}\omega _{ij}(X)c(e^{i})c(e^{j})\psi .
\label{spinconndef}
\end{equation}
\end{mydef}

\begin{myrmk}
\label{R-2-z} For simplicity, we use the same notation $\nabla $ for the
Tanaka-Webster connection acting on vector fields and the induced spinor
connection acting on spinors.
\end{myrmk}

\begin{mydef}
The CR/contact Dirac operator $\slashed{D}:\Gamma (\slashed{S})\rightarrow \Gamma
(\slashed{S})$ is defined by 
\begin{equation}
\slashed{D}\psi :=c(e^{i})\nabla _{e_{i}}\psi .  \label{Dirac}
\end{equation}
\end{mydef}

\begin{myrmk}
If we take a spin$^{c}$ struture, then 
\begin{equation*}
\slashed{D}^{c}:=c(e^{i})\nabla _{e_{i}}^{c}=c(\theta ^{\alpha })\nabla
_{Z_{\alpha }}^{c}+c(\theta ^{\bar{\alpha}})\nabla _{Z_{\bar{\alpha}}}^{c}=%
\sqrt{2}(\bar{\partial}_{b}+\bar{\partial}_{b}^{\ast })
\end{equation*}%
where $\nabla _{X}^{c}\psi =X\psi +\frac{1}{4}\omega
_{ij}(X)c(e^{i})c(e^{j})\psi +\frac{1}{2}\Gamma ^{det}(X)\psi $ and $\Gamma
^{det}=\omega _{\bar{\beta}}{}^{\bar{\beta}}$ is the connection on the
canonical line bundle $\Lambda ^{n}T^{\ast (0,1)}M$. See (4.5) in \cite%
{cheng2019heatkernel}.
\end{myrmk}

The following Weitzenb\"{o}ck-type formula is given in \cite%
{chengchiu2022positivemass5}.

\begin{mylem}
\label{Weitzen} (\cite{chengchiu2022positivemass5})%
\begin{equation}
\slashed{D}^{2}=\nabla^{\ast }\nabla+W-2\sum\limits_{\beta }c(e^{\beta
})c(e^{\beta +n})\nabla _{T}.  \label{2-4a}
\end{equation}
\end{mylem}

\begin{mypro}
\begin{equation}  \label{hardT}
\sum\limits_{\beta} c(e^{\beta})c(e^{\beta+n}) =\sum\limits_{\beta} \sqrt{-1}%
(\varepsilon_{\beta}\iota_{\beta}-\iota_{\beta}\varepsilon_{\beta}).
\end{equation}
and on $\Lambda^{k}T^{*(0,1)}M \subset \slashed{S}$, 
\begin{equation}  \label{hardT2}
\sum\limits_{\beta} c(e^{\beta})c(e^{\beta+n}) =\sqrt{-1} (2k-n)
I_{\Lambda^{k}}.
\end{equation}
\end{mypro}

\begin{proof}
\begin{align*}
\sum\limits_{\beta} c(e^{\beta})c(e^{\beta+n}) &= \sum\limits_{\beta}
(\varepsilon_{\beta}-\iota_{\beta})\sqrt{-1}(\varepsilon_{\beta}+\iota_{%
\beta}) \\
&=\sum\limits_{\beta} \sqrt{-1}(\varepsilon_{\beta}\iota_{\beta}-\iota_{%
\beta}\varepsilon_{\beta}).
\end{align*}
For the operator $(\varepsilon_{\beta}\iota_{\beta}-\iota_{\beta}%
\varepsilon_{\beta})$, note that it acts on the form $f\theta^{\overline{%
i_{1}}} \wedge...\wedge \theta^{\overline{i_{k}}}$ by multiplication with $%
\sqrt{-1}$ if the form contains $\theta^{\overline{\beta}}$, and by $-\sqrt{%
-1}$ otherwise. Therefore, the operator in (\ref{hardT}) acts on a $(0,k)$%
-form by multiplication with $\sqrt{-1} k-\sqrt{-1}(n-k)=\sqrt{-1}(2k-n)$.
\end{proof}

Hence, we obtain the following property.

\begin{mycor}
For $n=2$, we have, on $\Gamma (\slashed{S}^{odd})$ 
\begin{equation}
\Lambda :=\sum\limits_{\beta =1}^{n}c(e^{\beta })c(e^{\beta +n})=0
\label{2-7-1}
\end{equation}
\end{mycor}

\begin{proof}
For $n=2$, note that 
\begin{equation*}
\slashed{S}^{even}:=\Lambda ^{0,even}T^{\ast 0,1}M=\Lambda ^{0,0}T^{\ast
0,1}M\oplus \Lambda ^{0,2}T^{\ast 0,1}M,\quad \slashed{S}^{odd}:=\Lambda
^{0,1}T^{\ast 0,1}M.
\end{equation*}%
Thus, the operator in (\ref{hardT}) acting on $\Gamma (\slashed{S}^{odd})$
reduces to the action on $(0,1)$-forms. It acts as the multiplication by $%
\sqrt{-1}(2-2)=0$, i.e., the zero operator.
\end{proof}

\section{Hypersurface theory in Pseudohermitian manifolds\label{Sec3}}

Let $\Sigma $ be a hypersurface in $M$. Throughout this section, we work on
the nonsingular points of $\Sigma ,$ where $\xi $ is transversal to $T\Sigma
.$ Define 
\begin{equation*}
\xi _{1}:=\xi \cap T\Sigma ,\qquad \xi _{2}:=\xi _{1}\cap J\xi _{1}.
\end{equation*}%
By a dimension count, we obtain (as $\xi $ is transversal to $T\Sigma )$ 
\begin{equation*}
\dim \xi _{2}=2n-2,\qquad \dim \xi _{1}=2n-1.
\end{equation*}%
Locally, we take a local orthonormal frame $\{e_{1},e_{2},\ldots
,e_{n-1},e_{n+1},...,e_{2n-1}\}$ of $\xi _{2}$ with $e_{\beta +n}=Je_{\beta
} $ for $\beta =1,\ldots ,n-1$. Next, choose a unit vector $e_{n}$ lying in
the orthogonal complement of $\xi _{2}$ in $\xi _{1}$, define $%
e_{2n}:=Je_{n} $ as the CR normal vector of $\Sigma $ and $e_{n}$ as the
characteristic vector of $\Sigma $. Now there are two choices for $e_{2n}$
in view of the orientation. Now, for a bounded domain $\Omega$, we take an
inward unit normal vector $N$ on the boundary of $\Omega$ with respect to
the adapted metric $g= \theta^{2}+ L_{\theta}$. We choose $e_{n}$ such that $%
e_{2n}:=Je_{n}$ satisfies $\langle N,e_{2n} \rangle>0$. And we define the
area form $d\Sigma$ by $-e_{2n}\lrcorner dV$.

Since $\Sigma $ is nonsingular, there exists a function $\alpha $ on $M$
such that 
\begin{equation}
T+\alpha \,e_{2n}\in T\Sigma .  \label{dev}
\end{equation}
The vector $e_{0}:=T+\alpha e_{2n}$ and the function $\alpha $ are called,
respectively, the \emph{deviation vector} and the \emph{deviation function}
of $\Sigma $ in $M$.

\medskip

To describe the extrinsic geometry of $\Sigma$ in $M$ under the
pseudohermitian structure, we rewrite the Tanaka Webster structure equations
in real form. Define 
\begin{equation*}
Z_{\beta} := \frac{1}{2} (e_{\beta} - \sqrt{-1} e_{\beta+n}), \qquad
\theta^{\alpha} := e^{\beta} + \sqrt{-1} e^{\beta+n}.
\end{equation*}
In order to be compatible with the hypersurface $\Sigma$, we must modify the
coframe $\{\theta, e^{1}, \ldots, e^{2n}\}$. Since 
\begin{equation*}
T\Sigma = \langle e_{0}, e_{1}, \ldots, e_{2n-1} \rangle, \qquad N\Sigma =
\langle e_{2n} \rangle,
\end{equation*}
we choose the new coframe 
\begin{equation*}
T^{*}\Sigma = \langle \theta, e^{1}, \ldots, e^{2n-1} \rangle, \qquad
N^{*}\Sigma = \langle u^{2n} := e^{2n} - \alpha \theta \rangle.
\end{equation*}

\medskip

The original structure equations (\ref{realstreqn}) read as 
\begin{equation}  \label{str1}
\begin{cases}
\begin{aligned} d\theta \; &= \; \overline{\delta}_{ab}\, e^{a} \wedge e^{b}
, \\[6pt] de^{a} \; &= \; e^{b} \wedge \omega_{b}{}^{a} \;+\; \theta \wedge
\big(A^{a}{}_{b}\, e^{b} \big). \\[6pt] \end{aligned}%
\end{cases}%
\end{equation}

Let 
\begin{equation}
\begin{cases}
\begin{aligned} u^{0} \; &= \; \theta, \\[6pt] u^{i} \; &= \; e^{i}, \\[6pt]
u^{2n} \; &= \; e^{2n} -\alpha \theta. \\[6pt] \end{aligned}%
\end{cases}
\label{str2}
\end{equation}

Then 
\begin{eqnarray*}
du^{0} &=&d\theta \\
&=&\;\overline{\delta }_{ab}\,e^{a}\wedge e^{b} \\
&=&\;\overline{\delta }_{ij}\,e^{i}\wedge e^{j}+2e^{2n-1}\wedge
(u^{2n}+\alpha u^{0}) \\
&=&\;\overline{\delta }_{ij}\,u^{i}\wedge u^{j}-2\alpha u^{0}\wedge
u^{2n-1}-2u^{2n}\wedge u^{2n-1},
\end{eqnarray*}

\begin{align*}
du^{i}& =de^{i} \\
& =\;e^{b}\wedge \omega _{b}{}^{i}\;+\;\theta \wedge \big(A^{i}{}_{b}\,e^{b}%
\big) \\
& =e^{j}\wedge \Big({}^{\Sigma }\omega _{j}{}^{i}+\omega
_{j}{}^{i}(e_{2n})u^{2n}\Big)\;+\;e^{2n}\wedge \Big({}^{\Sigma }\omega
_{2n}{}^{i}+\omega _{2n}{}^{i}(e_{2n})u^{2n}\Big) \\
& \quad +\theta \wedge \big(A^{i}{}_{j}\,e^{j}+A^{i}{}_{2n}(u^{2n}+\alpha
u^{0})\big) \\
& =e^{j}\wedge {}^{\Sigma }\omega _{j}{}^{i}\;-u^{2n}\wedge \Big(\omega
_{j}{}^{i}(e_{2n})u^{j}\Big)\;+\;(u^{2n}+\alpha u^{0})\wedge \Big({}^{\Sigma
}\omega _{2n}{}^{i}+\omega _{2n}{}^{i}(e_{2n})u^{2n}\Big) \\
& \quad +u^{0}\wedge \big(A^{i}{}_{j}\,e^{j}\big)\;-u^{2n}\wedge
(A^{i}{}_{2n}u^{0}) \\
& =u^{j}\wedge {}^{\Sigma }\omega _{j}{}^{i}\;+u^{0}\wedge \big(%
A^{i}{}_{j}\,u^{j}+\alpha {}^{\Sigma }\omega _{2n}{}^{i}\big)\; \\
& \quad -u^{2n}\wedge \Big(\omega
_{j}{}^{i}(e_{2n})u^{j}+(A^{i}{}_{2n}+\alpha \omega
_{2n}{}^{i}(e_{2n}))u^{0}-{}^{\Sigma }\omega _{2n}{}^{i}\Big), \\
&
\end{align*}

and

\begin{align*}
du^{2n}& =d(e^{2n}-\alpha \theta ) \\
& =de^{2n}-\alpha d\theta -d\alpha \wedge \theta \\
& =u^{j}\wedge {}^{\Sigma }\omega _{j}{}^{2n}\;+u^{0}\wedge \big(%
A^{2n}{}_{j}\,u^{j}\big)\; \\
& \quad -u^{2n}\wedge \Big(\omega
_{j}{}^{2n}(e_{2n})u^{j}+(A^{2n}{}_{2n})u^{0}\Big) \\
& \quad -\alpha \Big(\,\overline{\delta }_{ij}\,u^{i}\wedge u^{j}-\alpha
u^{0}\wedge u^{2n-1}-u^{2n}\wedge u^{2n-1}\Big) \\
& \quad -(e_{a}\alpha )u^{a}\wedge u^{0} \\
& =u^{j}\wedge \Big({}^{\Sigma }\omega _{j}{}^{2n}-\alpha \overline{\delta }%
_{jk}u^{k}\Big)\;+u^{0}\wedge \Big(A^{2n}{}_{j}+(e_{j}\alpha )+\alpha
^{2}\delta _{j}^{2n-1}\Big)u^{j}\; \\
& \quad +u^{2n}\wedge \Big(\alpha u^{2n-1}-\omega
_{j}{}^{2n}(e_{2n})u^{j}-(A^{2n}{}_{2n}+e_{2n}\alpha )u^{0}\Big). \\
&
\end{align*}

Therefore, the structure equations take the form 
\begin{equation}
\begin{cases}
\begin{aligned} du^{0} \; &= \; \overline{\delta}_{ij}\, u^{i} \wedge u^{j}
\;-\; 2\alpha\, u^{0} \wedge u^{2n-1} \;-\; 2 u^{2n} \wedge u^{2n-1},
\\[6pt] du^{i} \; &= \; u^{j} \wedge {}^{\Sigma}\omega_{j}{}^{i} \;+\; u^{0}
\wedge \big(A^{i}{}_{j}\, u^{j} + \alpha\,
{}^{\Sigma}\omega_{2n}{}^{i}\big)\\ &\quad+\; u^{2n} \wedge \Big(
{}^{\Sigma}\omega_{2n}{}^{i} -\omega_{j}{}^{i}(e_{2n})u^{j}
-(A^{i}{}_{2n}+\alpha \omega_{2n}{}^{i}(e_{2n}))\, u^{0} \Big), \\[6pt]
du^{2n} \; &= \; u^{j} \wedge \Big( {}^{\Sigma}\omega_{j}{}^{2n} - \alpha\,
\overline{\delta}_{jk}\, u^{k} \Big) \;+\; u^{0} \wedge \Big(A^{2n}{}_{i} +
(e_{i}\alpha) + \alpha^{2}\, \delta^{2n-1}_{i}\Big)u^{i} \\[3pt] &\quad\;+\;
u^{2n} \wedge \Big( \alpha\,u^{2n-1} \;-\; \omega_{j}{}^{2n}(e_{2n}) \,
u^{j} \;-\; (A^{2n}{}_{2n} + (e_{2n} \alpha))\, u^{0} \Big). \end{aligned}%
\end{cases}
\label{hyperstreqn}
\end{equation}%
where 
\begin{equation}
^{\Sigma }\omega _{j^{\prime }}{}^{i}:=\omega _{j^{\prime }}{}^{i}-\omega
_{j^{\prime }}{}^{i}(e_{2n})u^{2n}.  \label{hcon}
\end{equation}

\medskip

Since $T\Sigma$ is integrable, we have $du^{2n} \wedge u^{2n} = 0$, which
yields the following integrability condition: 
\begin{equation}  \label{hyperintcondi}
\begin{cases}
\begin{aligned} & \omega_{j}{}^{2n}(e_{i}) \;-\; \alpha\,
\overline{\delta}_{ji} = \; \omega_{i}{}^{2n}(e_{j}) \;-\; \alpha\,
\overline{\delta}_{ij}, \\[6pt] & A^{2n}{}_{i} \;+\; (e_{i}\alpha) \;-\;
\omega_{i}{}^{2n}(e_{0}) \;+\; \alpha^{2}\, \delta^{2n-1}_{i} = \; 0.
\end{aligned}%
\end{cases}%
\end{equation}

\begin{myrmk}
The wedge term of $u^{2n}$, namely 
\begin{equation*}
\alpha\, u^{2n-1} - \omega_{j}{}^{2n}(e_{2n})\, u^{j} - \bigl(A^{2n}{}_{2n}
+ (e_{2n}\alpha)\bigr)\theta,
\end{equation*}
vanishes if and only if the chosen frame satisfies the
Carnot--Carath\'eodory frame condition in \cite{kao2025thesis}.
\end{myrmk}

\medskip

\subsection{Pseudohermitian connection and second fundamental form}

We first define the (pseudohermitian) second fundamental form: for $1\leq
i,j\leq 2n-1$

\begin{equation}
h_{ji}=h(e_{j},e_{i}):= -\langle \nabla_{e_{i}} e_{2n}, e_{j}\rangle=
\omega_{j}{}^{2n}(e_{i})
\end{equation}

\medskip

Using the integrability condition \eqref{hyperintcondi}, we can define the
hypersurface connection and the second fundamental form on $\Sigma$.

\begin{mydef}
Define the \emph{hypersurface connection} $\nabla ^{\Sigma }$ to be the
connection $1$-form ${}^{\Sigma }\omega _{i}{}^{j}$ on $T\Sigma $ (see (\ref%
{hcon})). The \emph{second fundamental form} $A:T\Sigma \rightarrow T\Sigma $
is defined by 
\begin{equation*}
Ae_{i}=\sum\limits_{j=1}^{2n-1}h_{ji}e_{j},\qquad \nabla _{X}e_{2n}=-AX.
\end{equation*}%
By \eqref{hyperintcondi}, we may also define the symmetrization of the
second fundamental form (see \cite{CCHY18}) by 
\begin{equation*}
\Tilde{A}e_{i}=\Tilde{h}_{ji}e_{j}=\bigl(h_{ji}-\alpha \overline{\delta }%
_{ji}\bigr)e_{j}.
\end{equation*}%
Moreover, let $J^{\prime }X=JX-\langle X,e_{n}\rangle e_{2n}$, then we have 
\begin{equation*}
\Tilde{A}X=-\nabla _{X}e_{2n}+\alpha J^{\prime }X.
\end{equation*}%
The $p$-mean curvature $H$ is given by 
\begin{equation}
H=\sum\limits_{i=1}^{2n-1} h_{ii}=\sum\limits_{i=1}^{2n-1} \Tilde{h}_{ii}.
\label{pmc}
\end{equation}%
A hypersurface $\Sigma $ is called $p$-minimal if its $p$-mean curvature $%
H=0.$
\end{mydef}

Thus, for $X,Y\in T\Sigma $, we have 
\begin{equation*}
\nabla _{X}Y=\nabla _{X}^{\Sigma }Y+\langle Y,AX\rangle e_{2n}.
\end{equation*}

\begin{myrmk}
We have%
\begin{equation*}
\nabla _{e_{i}}e_{2n}=\omega _{2n}{}^{j}(e_{i})e_{j}=-\omega
_{j}{}^{2n}(e_{i})e_{j}=-h_{ji}e_{j}
\end{equation*}%
and 
\begin{equation*}
\nabla _{e_{i}}e^{2n}=-\omega _{j}{}^{2n}(e_{i})e^{j}=-h_{ji}e^{j}
\end{equation*}

The hypersurface area form $d\Sigma$ is defined by 
\begin{equation}  \label{dsigma}
d\Sigma:=-e_{2n} \lrcorner dV
\end{equation}
Let $X=c_{i}e^{i}$. We also write 
\begin{equation*}
-\nabla _{X}e^{2n}=c_{i}h_{ji}e^{j}=AX,
\end{equation*}%
where the last equality defines the shape operator $A$.
\end{myrmk}

\subsection{Spin structures on hypersurfaces in a CR manifold}

Let $\Omega $ be an open domain in $(M,J,\theta )$ with smooth nonsingular
boundary $\Sigma =\partial \Omega $. We define the hypersurface spinor
bundle $\slashed{S}^{\Sigma }$ by 
\begin{equation}
\slashed{S}^{\Sigma }:=\slashed{S}|_{\Sigma },  \label{3-9a}
\end{equation}
\noindent together with the Clifford algebra homomorphism 
\begin{equation*}
c^{\Sigma }(\cdot ):T^{\ast }\Sigma \longrightarrow \mathrm{End}%
(\slashed{S}^{\Sigma }),\qquad c^{\Sigma }(v)=\mathrm{\ }c(v)c(e^{2n}).
\end{equation*}


\begin{mypro}
\label{P-3-1} The connection induced on $\slashed{S}^{\Sigma }$ satisfies%
\begin{equation*}
\nabla _{X}^{\Sigma }=\nabla _{X}-\tfrac{1}{2}c^{\Sigma }(AX).
\end{equation*}
\end{mypro}

\begin{proof}
Let $i^{\prime },j^{\prime }$ run from $1$ to $2n-1$. We compute
\begin{align*}
\nabla_{X}\psi =& X\psi + \frac{1}{4} \omega_{ij}(X)c(e^{i})c(e^{j})\psi \\
=& X\psi + \frac{1}{4} \omega_{i^{\prime }j^{\prime }}(X)c(e^{i^{\prime
}})c(e^{j^{\prime }})\psi \\
&+\frac{1}{4} \omega_{(2n)j}(X)c(e^{2n})c(e^{j})\psi+\frac{1}{4}
\omega_{i(2n)}(X)c(e^{i})c(e^{2n})\psi \\
=& \nabla^{\Sigma}_{X}\psi +\frac{1}{4}c(e^{2n})c(\nabla_{X}e^{2n})\psi-%
\frac{1}{4}c(\nabla_{X}e^{2n})c(e^{2n})\psi \\
=& \nabla^{\Sigma}_{X}\psi +\frac{1}{2} c^{\Sigma}(AX) \psi .
\end{align*}
\end{proof}

The corresponding hypersurface Dirac operator is defined as 
\begin{equation}
\slashed{D}^{\Sigma }:=\sum_{i^{\prime }=1}^{2n-1}c^{\Sigma }(e^{i^{\prime
}})\nabla _{e_{i^{\prime }}}^{\Sigma }.  \label{hDirac}
\end{equation}

By a direct computation, one obtains

\begin{align*}
\slashed{D}^{\Sigma} =& \sum_{i^{\prime }=1}^{2n-1} c^{\Sigma}(e^{i^{\prime
}})\nabla^{\Sigma}_{e_{i^{\prime }}} \\
=& \sum_{i^{\prime }=1}^{2n-1} c(e^{i^{\prime }})c(e^{2n})
[\nabla_{e_{i^{\prime }}}- \tfrac{1}{2}c(Ae_{i^{\prime }})c(e^{2n})] \\
=& \sum_{i^{\prime }=1}^{2n-1} -c(e^{2n})c(e^{i^{\prime
}})\nabla_{e_{i^{\prime }}} -\frac{1}{2}c(e^{i^{\prime }})c(h_{j^{\prime
}i^{\prime }}e^{j^{\prime }}) \\
=& \sum_{i^{\prime }=1}^{2n-1} -c(e^{2n})[\slashed{D}-c(e^{2n})\nabla_{e_{2n}}]
+\frac{1}{2}H -\frac{1}{2}\sum\limits_{i^{\prime }<j^{\prime }}
c(e^{i^{\prime }})c(e^{j^{\prime }})[h_{j^{\prime }i^{\prime }}-h_{i^{\prime
}j^{\prime }}] \\
=& -c(e^{2n}) \slashed{D} \;-\; \nabla_{e_{2n}} \;+\; \tfrac{1}{2} H \;+\;
\alpha \sum_{\beta=1}^{n-1} c(e^{\beta})c(e^{\beta+n}).
\end{align*}
Recall 
\begin{equation}
\Lambda :=\sum_{\beta =1}^{n}c(e^{\beta })c(e^{\beta +n})  \label{Cl2}
\end{equation}%
and define 
\begin{equation}
\Lambda ^{\Sigma }:=\sum_{\beta =1}^{n-1}c(e^{\beta })c(e^{\beta +n}).
\label{hCl2}
\end{equation}
Therefore, we have \textbf{the hypersurface Dirac operator}

\begin{equation}  \label{diraceqnhyper}
\slashed{D}^{\Sigma} = -c(e^{2n}) \slashed{D} \;-\; \nabla_{e_{2n}} \;+\; \tfrac{1%
}{2} H \;+\; \alpha \Lambda^{\Sigma}.
\end{equation}
However, the operator $\slashed{D}^{\Sigma}$ is not self-adjoint. For this
reason, we introduce a modified hypersurface Dirac operator $\Tilde{\slashed{D}}%
^{\Sigma}$.

\begin{mypro}
The following properties hold:

\begin{itemize}
\item For any spinors $\psi, \varphi$, one has 
\begin{equation}  \label{hyperantisym}
\langle c^{\Sigma}(e^{i}) \psi, \varphi \rangle = - \langle \psi,
c^{\Sigma}(e^{i}) \varphi \rangle.
\end{equation}

\item The Clifford action satisfies the compatibility relation 
\begin{equation}  \label{cliffchainrule}
\begin{aligned} \nabla^{\Sigma}_{e_{\jp}}(c^{\Sigma}(e^{\ip})\psi ) &=
c^{\Sigma}(\nabla^{\Sigma}_{e_{\jp}} e^{\ip}) \psi + c^{\Sigma}(e^{\ip})
\nabla^{\Sigma}_{e_{\jp}} \psi .\end{aligned}
\end{equation}

\item The hypersurface connection is compactible with the hypersurface
metric: i.e. $e_{i^{\prime }}\langle \psi, \varphi \rangle= \langle
\nabla^{\Sigma}_{e_{i^{\prime }}} \psi,\varphi \rangle+\langle \psi,
\nabla^{\Sigma}_{e_{i^{\prime }}}\varphi \rangle$.
\end{itemize}
\end{mypro}

\begin{proof}
For (\ref{hyperantisym}), we have 
\begin{align*}
\langle c^{\Sigma}(e_{i^{\prime }}) \psi, \varphi \rangle =& \langle
c(e^{i^{\prime }})c(e^{2n}) \psi, \varphi \rangle \\
=& \langle \psi, c(e^{2n})c(e^{i^{\prime }})\varphi \rangle \\
=& \langle \psi, -c^{\Sigma}(e_{i^{\prime }}) \varphi \rangle .
\end{align*}
For (\ref{cliffchainrule}), we have
\begin{align*}
& \nabla _{e_{j^{\prime }}}^{\Sigma }(c^{\Sigma }(e^{i^{\prime }})\psi ) \\
=& [\nabla _{e_{j^{\prime }}}-\tfrac{1}{2}c(Ae_{j^{\prime
}})c(e^{2n})](c(e^{i^{\prime }})c(e^{2n})\psi ) \\
=& c(\nabla _{e_{j^{\prime }}}e^{i^{\prime }})c(e^{2n})\psi +c(e^{i^{\prime
}})c(\nabla _{e_{j^{\prime }}}e^{2n})\psi +c(e^{i^{\prime }})c(e^{2n})\nabla
_{e_{j^{\prime }}}\psi -\frac{1}{2}c(Ae_{j^{\prime }})c(e^{i^{\prime }})\psi
\\
=& c(\nabla _{e_{j^{\prime }}}^{\Sigma }e^{i^{\prime }}+h_{i^{\prime
}j^{\prime }}e^{2n})c(e^{2n})\psi +c(e^{i^{\prime }})c(-Ae_{j^{\prime
}})\psi -\frac{1}{2}c(Ae_{j^{\prime }})c(e^{i^{\prime }})\psi \\
& +c(e^{i^{\prime }})c(e^{2n})[\nabla _{e_{j^{\prime }}}^{\Sigma }+\frac{1}{2%
}c^{\Sigma }(Ae_{j^{\prime }})]\psi \\
=& c^{\Sigma }(\nabla _{e_{j^{\prime }}}^{\Sigma }e^{i^{\prime }})\psi
+c^{\Sigma }(e^{i^{\prime }})\nabla _{e_{j^{\prime }}}^{\Sigma }\psi
-h_{i^{\prime }j^{\prime }}\psi -\tfrac{1}{2}\big(c(e^{i^{\prime
}})c(Ae_{j^{\prime }})+c(Ae_{j^{\prime }})c(e^{i^{\prime }})\big)\psi \\
=& c^{\Sigma }(\nabla _{e_{j^{\prime }}}^{\Sigma }e^{i^{\prime }})\psi
+c^{\Sigma }(e^{i^{\prime }})\nabla _{e_{j^{\prime }}}^{\Sigma }\psi
-h_{i^{\prime }j^{\prime }}\psi -\tfrac{1}{2}h_{k^{\prime }j^{\prime }}\big(%
c(e_{i^{\prime }})c(e_{k^{\prime }})+c(e_{k^{\prime }})c(e_{i^{\prime }})%
\big)\psi \\
=& c^{\Sigma }(\nabla _{e_{j^{\prime }}}^{\Sigma }e_{i^{\prime }})\psi
+c^{\Sigma }(e_{i^{\prime }})\nabla _{e_{j^{\prime }}}^{\Sigma }\psi .
\end{align*}
The third property is straightforward.
\end{proof}

\begin{mypro}
\label{selfadj} We have the following equality

\begin{equation*}
\oint\limits_{\Sigma} \langle \slashed{D}^{\Sigma} \psi, \varphi \rangle
d\Sigma =\oint_{\Sigma} [\langle \psi, \slashed{D}^{\Sigma} \varphi \rangle +
\langle \psi, 2\alpha c^{\Sigma}(e^{n}) \varphi \rangle ]d\Sigma.
\end{equation*}
\end{mypro}

\begin{proof}
Using the fact that 
\begin{align*}
&\oint\limits_{\Sigma} d(\langle \psi, -c^{\Sigma}(e^{i^{\prime }}) \varphi
\rangle e_{i^{\prime }} \mathbin{\lrcorner} d\Sigma) \\
=& \oint\limits_{\Sigma} e_{i^{\prime }}\langle \psi,
-c^{\Sigma}(e_{i^{\prime }}) \varphi \rangle d\Sigma + \oint\limits_{\Sigma}
\langle \psi, c^{\Sigma}(\nabla^{\Sigma}_{e_{i^{\prime }}} e^{i^{\prime }})
\varphi \rangle d\Sigma - \oint\limits_{\Sigma} \langle \psi,
c^{\Sigma}(2\alpha e^{n}) \varphi \rangle d\Sigma
\end{align*}

We have 
\begin{align*}
&\oint\limits_{\Sigma} \langle \slashed{D}^{\Sigma} \psi, \varphi \rangle
d\Sigma \\
=&\oint\limits_{\Sigma} \langle \nabla^{\Sigma}_{e_{i^{\prime }}} \psi,
-c^{\Sigma}(e^{i^{\prime }})\varphi \rangle d\Sigma \\
=& \oint\limits_{\Sigma} e_{i^{\prime }}\langle
\psi,-c^{\Sigma}(e^{i^{\prime }})\varphi \rangle d\Sigma+
\oint\limits_{\Sigma} \langle \psi, \nabla^{\Sigma}_{e_{i^{\prime
}}}(c^{\Sigma}(e^{i^{\prime }})\varphi)\rangle d\Sigma \\
=&\oint\limits_{\Sigma} d(\langle \psi, -c^{\Sigma}(e^{i^{\prime }}) \varphi
\rangle e_{i^{\prime }} \mathbin{\lrcorner} d\Sigma) +\oint\limits_{\Sigma}
\langle \psi,\slashed{D}^{\Sigma} \varphi \rangle d\Sigma
+\oint\limits_{\Sigma} \langle \psi, c^{\Sigma}(2\alpha e^{n}) \varphi
\rangle d\Sigma \\
=& \oint_{\Sigma} [\langle \psi, \slashed{D}^{\Sigma}\varphi \rangle + \langle
\psi, 2\alpha c^{\Sigma}(e^{n}) \varphi \rangle ]d\Sigma.
\end{align*}
\end{proof}

We define the \textbf{modified hypersurface Dirac operator} as follows: 
\begin{equation}
\Tilde{\slashed{D}}^{\Sigma}:=\slashed{D}^{\Sigma}+\alpha c^{\Sigma }(e^{n}).
\label{mhDirac}
\end{equation}

\begin{mylem}
\label{selfadjforhyperdirac} This operator $\Tilde{\slashed{D}}^{\Sigma}$ is
self-adjoint in $L^{2}(\Sigma ,\slashed{S}^{\Sigma })$.
\end{mylem}

\begin{proof}
It follows from Proposition \ref{selfadj}.
\end{proof}

By (\ref{diraceqnhyper}), we have a formula for the modified hypersurface
Dirac operator

\begin{equation}  \label{moddiraceqnhyper}
\Tilde{\slashed{D}}^{\Sigma} = -c(e^{2n}) \slashed{D} \;-\; \nabla_{e_{2n}} \;+\; 
\tfrac{1}{2} H \;+\; \alpha \Lambda.
\end{equation}


\section{Reilly-type inequality in CR manifolds: proof of theorem \protect
\ref{thm1}\label{Sec4}}

Let $\Omega$ be an open domain in $(M,J,\theta)$ with (smooth) boundary $%
\Sigma = \partial \Omega$. We begin by considering the following two $1$%
-forms: 
\begin{align}  \label{oneforms}
w_{1}(X) &= \langle c(X^{*})\slashed{D} \psi, \psi \rangle, \\
w_{2}(X) &= \langle \nabla_{X}\psi, \psi \rangle.
\end{align}
The divergences of these forms are given by 
\begin{align*}
(\mathrm{div}_{b}\,w_{1}) &= \langle \slashed{D}^{2} \psi, \psi \rangle - \|
\slashed{D} \psi \|^{2}, \\
(\mathrm{div}_{b}\,w_{2}) &= -\langle \nabla^{*}\nabla\psi, \psi \rangle +
\| \nabla \psi \|^{2}.
\end{align*}

Applying Stokes' theorem and the Weitzenb\"{o}ck-type formula for CR Dirac
operator (cf. Lemma \ref{Weitzen}), we obtain 
\begin{equation}  \label{divformula}
\begin{aligned} &\oint_{\Sigma} \langle -c(e^{2n}) \slashed{D}\psi -
\nabla_{e_{2n}} \psi, \psi \rangle \, d\Sigma \\ &= \int_{\Omega}
\Big[-|\slashed{D} \psi|^{2} + W |\psi|^{2} + |\nabla \psi|^{2} -2 \big\langle
\sum_{\beta=1}^{n} c(e^{\beta })c(e^{\beta+n}) \nabla_{T} \psi, \psi
\big\rangle \Big] dV. \end{aligned}
\end{equation}

\begin{myrmk}
\label{R-4-1} Applying Stokes' theorem to a quantity smooth up to the
boundary in a smooth domain gives an integral of a smooth quantity over the
boundary $\Sigma $. As the set of singular points in $\Sigma $ is a
submanifold, hence of measure $0$ by \cite[Theorem D]{CHY07}, we may choose
a suitable frame on the nonsingular part to express the boundary integrand.
So the formula (\ref{divformula}) and later formulas hold true if one
considers the boundary integrals only over the set of nonsingular points.
\end{myrmk}

\medskip This yields the \emph{complex Reilly formula} in the CR setting: 
\begin{equation}  \label{cpxreillyeqn}
\begin{aligned} &\oint_{\Sigma} \langle \Tilde{\slashed{D}}^{\Sigma} \psi,
\psi \rangle \, d\Sigma - \oint_{\Sigma} \frac{1}{2} H |\psi|^{2} \, d\Sigma
- \oint_{\Sigma} \alpha \langle \Lambda \psi,\psi \rangle \, d\Sigma\\ &=
\int_{\Omega} \Big[ -|\slashed{D} \psi|^{2} + W |\psi|^{2} + |\nabla \psi|^{2}
- 2\big\langle \Lambda \nabla_{T} \psi, \psi \big\rangle \Big] dV.
\end{aligned}
\end{equation}

\noindent Splitting the complex Reilly formula (\ref{cpxreillyeqn}) into the
real part and the imaginary part, we have

1. Real part (CR version of Reilly formula): 
\begin{equation}  \label{reillyeqn}
\begin{aligned} &\oint_{\Sigma} \langle \Tilde{\slashed{D}}^{\Sigma} \psi,
\psi \rangle \, d\Sigma - \frac{1}{2}\oint_{\Sigma} H |\psi|^{2} \, d\Sigma
\\ &= \int_{\Omega} \Big[ -|\slashed{D} \psi|^{2} + W |\psi|^{2} + |\nabla
\psi|^{2} - 2\re[ \big\langle \Lambda \nabla_{T} \psi, \psi \big\rangle]
\Big] dV. \end{aligned}
\end{equation}

which completes the proof of 1) part of Theorem \ref{thm1}.

2. Imaginary part: 
\begin{equation}  \label{imreillyeqn}
\begin{aligned} &\oint_{\Sigma} \alpha \langle \Lambda \psi,\psi \rangle \,
d\Sigma = \sqrt{-1}\int_{\Omega} 2\im[ \big\langle \Lambda \nabla_{T} \psi,
\psi \big\rangle] dV. \end{aligned}
\end{equation}

In fact, we are going to show that (\ref{imreillyeqn}) holds ``trivially".
We need a proposition for $\Lambda$.

\begin{mypro}
It holds that $[\nabla, \Lambda]=0 $ on $\Gamma(\bar{\Omega},\slashed{S})$.
\end{mypro}

\begin{proof}
Let $X\in T\bar{\Omega}$ and $\psi \in \Gamma(\bar{\Omega},\slashed{S})$. Then we have
\begin{align*}
&\nabla_{X} \Lambda \psi -\Lambda \nabla_{X} \psi \\
=&\sum\limits_{\alpha=1}^{n} c(\nabla_{X}e^{\alpha})c(e^{n+\alpha}) +
c(e^{\alpha})c(\nabla_{X}e^{n+\alpha}) \psi \\
=& \sum\limits_{\alpha=1}^{n} \sum\limits_{j=1}^{2n}
-\omega_{j}{}^{\alpha}(X)c(e^{j})c(e^{n+\alpha})-\omega_{j}{}^{n+%
\alpha}c(e^{\alpha})c(e^{j}) \\
=& -\sum\limits_{\alpha=1}^{n} \sum\limits_{\beta=1, \beta\neq \alpha}^{n} [
\omega_{\beta}{}^{\alpha}(X)c(e^{\beta})c(e^{n+\alpha}) \\
&+ \omega_{n+\beta}{}^{\alpha}(X)c(e^{n+\beta})c(e^{n+\alpha}) \\
&+ \omega_{\beta}{}^{n+\alpha}(X)c(e^{\alpha})c(e^{\beta}) \\
&+ \omega_{n+\beta}{}^{n+\alpha}(X)c(e^{\alpha})c(e^{n+\beta}) ]=0 \\
\end{align*}
The first two equalities follow from the definition while the third equality follows from (\ref{realstreqn}).
\end{proof}

\begin{mycor}
The formula (\ref{imreillyeqn}) always holds.
\end{mycor}

\begin{proof} We compute
\begin{align*}
&\text{RHS of (\ref{imreillyeqn})} \\
=& \sqrt{-1}\int_{\Omega} 2\im[ \big\langle \Lambda \nabla_{T} \psi, \psi
\big\rangle] dV \\
=& c_{n}\int_{\Omega} (\big\langle \nabla_{T} \Lambda \psi, \psi %
\big\rangle + \big\langle \Lambda \psi, \nabla_{T} \psi \big\rangle) \theta
\wedge (d\theta)^{n} \\
=& c_{n}\int_{\Omega} d( \big\langle \Lambda \psi, \psi \big\rangle %
(d\theta)^{n}) \\
=&  c_{n}\oint_{\Sigma} \big\langle \Lambda \psi, \psi \big\rangle %
(d\theta)^{n} \\
=&  \oint_{\Sigma} \big\langle \Lambda \psi, \psi \big\rangle e^{1}\wedge
e^{2} \wedge ... \wedge e^{2n-1} \wedge e^{2n} \\
=& \oint_{\Sigma} \big\langle \Lambda \psi, \psi \big\rangle e^{1}\wedge
e^{2} \wedge ... \wedge e^{2n-1} \wedge (\alpha \theta) \\
=& \oint_{\Sigma} \alpha \big\langle \Lambda \psi, \psi \big\rangle d\Sigma
= \text{LHS of (\ref{imreillyeqn})}. \\
\end{align*}
\end{proof}

Define the twistor operator by 
\begin{equation*}
P_{X}\psi := \nabla_{X} \psi + \tfrac{1}{2n} c(X)\slashed{D} \psi.
\end{equation*}
Then one has the standard identity 
\begin{equation}  \label{twist}
\|\nabla \psi\|^{2} = \|P\psi\|^{2} + \tfrac{1}{2n} \|\slashed{D}\psi\|^{2}.
\end{equation}
Substituting (\ref{twist}) into \eqref{reillyeqn}, we obtain the inequality (%
\ref{1-2a}) and complete the proof of 2) part of theorem \ref{thm1}.



\section{Pseudohermitian mass: proof of Theorem \protect\ref{pmassthm} and
Corollary \textbf{\protect\ref{pmasscor}\label{Sec5}}}

For the definition of asymptotically flat pseudohermitian manifold and $p$%
-mass, see \cite{chengchiu2022positivemass5}. These notions play an
important role in solving the Yamabe minimizer problem on CR manifolds.

\begin{mypro}
\label{P-5-1} (\cite{chengchiu2022positivemass5}) Let $M$ be an
asymptotically flat pseudohermitian manifold of dimension 5. Then there
exists $\psi $ $\in $ $\Gamma (M,\slashed{S}^{odd})$ ($\slashed{S}^{even}$ in \cite%
{chengchiu2022positivemass5} due to a different way to define the Clifford
action)$,$ satisfying the following properties:\newline
1. $\slashed{D}^{2}\psi =0$ on $M$\newline
2. $\slashed{D}\psi =O(\rho ^{-6+\varepsilon })$ near the asymptotic end\newline
3. $\psi =\psi _{0}+\psi _{-4+\varepsilon }$, where $\psi _{0}$ is a
constant spinor with norm $1$ near the asymptotic end and $\psi
_{-4+\varepsilon }=O(\rho ^{-4+\varepsilon }).$
\end{mypro}

Let $L^{i}:=\delta^{ij}\nabla _{e_{j}}+c(e^{i})\slashed{D}$ and $w:=\langle
\lbrack c(e^{i}),c(e^{j})]\psi ,\varphi \rangle e_{i}\lrcorner
e_{j}\lrcorner dV$.

\begin{mypro}
\label{selfadjpropL} With the above notations, we have 
\begin{equation*}
dw=-4(\langle L^{i}\psi ,\varphi \rangle -\langle \psi ,L^{i}\varphi \rangle
)e_{i}\lrcorner dV+2\overline{\delta }_{ji}\langle \lbrack
c(e^{i}),c(e^{j})]\psi ,\varphi \rangle T\lrcorner dV.
\end{equation*}
\end{mypro}

\begin{proof}
    In \cite{chengchiu2022positivemass5}, we have the fact that
    \begin{equation*}
        L^{i}:=\delta^{ij}\nabla _{e_{j}}+c(e^{i})\slashed{D}=\frac{1}{2}[c(e^{i}),c(e^{j})]\nabla_{e_{j}}.
    \end{equation*}
    Then we compute
    \begin{align*}
        d\omega =& d(\langle [c(e^{i}),c(e^{j})]\psi ,\varphi \rangle e_{i}\lrcorner e_{j}\lrcorner dV)\\
        =& e_{k}\langle [ c(e^{i}),c(e^{j})]\psi ,\varphi \rangle e^{k} \w (e_{i}\lrcorner e_{j}\lrcorner dV)\\
        &+ \langle [ c(e^{i}),c(e^{j})]\psi ,\varphi \rangle d(e_{i}\lrcorner e_{j}\lrcorner dV)\\
        =& \langle [c(e^{i}),c(e^{j})] \nabla_{e_{k}}\psi ,\varphi \rangle (\delta_{i}^{k} e_{j}\lrcorner dV -\delta_{k}^{j} e_{i}\lrcorner dV)\\
        &+ \langle [c(e^{i}),c(e^{j})] \psi , \nabla_{e_{k}}\varphi \rangle  (\delta_{i}^{k} e_{j}\lrcorner dV -\delta_{k}^{j} e_{i}\lrcorner dV)\\
        &+ \langle [c(e^{i}),c(e^{j})]\psi ,\varphi \rangle  d\theta \w (e_{i}\lrcorner e_{j}\lrcorner (T\lrcorner dV))\\
        =& -4(\langle L^{i}\psi ,\varphi \rangle -\langle \psi ,L^{i}\varphi \rangle
        )e_{i}\lrcorner dV +2\overline{\delta }_{ji}\langle \lbrack
        c(e^{i}),c(e^{j})]\psi ,\varphi \rangle T\lrcorner dV.
    \end{align*}
\end{proof}

\begin{myrmk}
1. The difference between Riemannian and CR cases is the $\theta$ part
because $d\theta$ matters when $\omega$ is a less than $2n-$ form.\newline
2. If we take the hypersurface frame ($e_{2n}$ is viewed as a CR unit normal
vector field of $\Sigma$), Proposition \ref{selfadjpropL} implies 
\begin{align*}
\oint\limits_{\Sigma} [\langle (L^{2n}-\alpha \Lambda) \psi, \varphi \rangle
- \langle \psi, (L^{2n}-\alpha \Lambda)\varphi \rangle] d\Sigma=0.
\end{align*}
Therefore it implies that $L^{2n}-\alpha \Lambda $ is self-adjoint in $%
L^{2}(\Sigma ,\slashed{S})$.
\end{myrmk}

\begin{myrmk}
In (4.17) of \cite{chengchiu2022positivemass5}, the authors miss the terms $%
\alpha \Lambda$ in self-adjopintness of $L$. However, as $\alpha \Lambda$ is
purely imaginary, it does not affect the real part of the integral there.
\end{myrmk}

By equation (\ref{divformula}), we want to estimate $\re\oint_{S_{K}}
\langle L^{i}\psi,\psi \rangle e_{i}\lrcorner dV$ as $K$ goes to infinity.
We have 
\begin{equation}  \label{pmassinf}
\begin{aligned} &\re \oint\limits_{S_{K}} \langle L^{i}\psi,\psi \rangle
e_{i}\lrcorner dV \\ =&\re \{\oint\limits_{S_{K}} \langle
\nabla_{e_{i}}\psi,\psi \rangle e_{i}\lrcorner dV \\ +& \oint\limits_{S_{K}}
\langle c(e^{i})\slashed{D}\psi,\psi \rangle e_{i}\lrcorner dV \} \end{aligned}
\end{equation}

By Proposition \ref{P-5-1}, we have $\slashed{D}\psi =O(\rho^{-6+\varepsilon})$%
. Thus $\langle c(e^{i})\slashed{D}\psi,\psi \rangle e_{i}\lrcorner dV
=O(\rho^{-1+\varepsilon})$ and $\re\oint_{S_{K}} \langle
c(e^{i})\slashed{D}\psi,\psi \rangle e_{i}\lrcorner dV \rightarrow 0$ as $K$
goes to infinity.

In \cite{chengchiu2022positivemass5}, the authors showed that 
\begin{equation}
\re \{\oint\limits_{S_{K}} \langle \nabla_{e_{i}}\psi,\psi \rangle
e_{i}\lrcorner dV \} \rightarrow cm,
\end{equation}
as $K$ goes to infinity for some positive constant real number $c$, see
equation(4.12),(4.20), Lemma 3.5 and 4.4 in \cite{chengchiu2022positivemass5}%
.

\begin{proof}
\textbf{(of Theorem \ref{pmassthm})} 
By taking the real part of (\ref{divformula}), we have
    \begin{align*} 
        &\re \oint_{\Sigma} \langle c(e^{i}) \slashed{D}\psi +
        \nabla_{e_{i}} \psi, \psi \rangle \, e_{i}\lrcorner dV \\ &= \int_{\Omega}
        \Big[-|\slashed{D} \psi|^{2} + W |\psi|^{2} + |\nabla \psi|^{2} -2 \re
        \big\langle \sum_{\beta=1}^{n} c(e^{\beta })c(e^{\beta+n}) \nabla_{T} \psi,
        \psi \big\rangle \Big] dV. 
    \end{align*}
Take the domain $\Omega _{K}$ in $M,$ bounded by a Heisenberg ball $S_{K}$ in the asymptotic end and an inner boundary $\Sigma$. For asymptotic end part, using equation (\ref{pmassinf}) to get the mass part. For inner part, take a hypersurface frame with inner CR unit normal vector field $e_{2n}$ and use the CR Reilly formula (\ref{reillyeqn}). Then we have
\begin{align} 
        &\int_{\Omega}
        \Big[-|\slashed{D} \psi|^{2} + W |\psi|^{2} + |\nabla \psi|^{2} -2 \re
        \big\langle \sum_{\beta=1}^{n} c(e^{\beta })c(e^{\beta+n}) \nabla_{T} \psi,
        \psi \big\rangle \Big] dV \label{5-2} \\ 
        &= \re \oint_{S_{K}} \langle c(e^{i}) \slashed{D}\psi +
        \nabla_{e_{i}} \psi, \psi \rangle \, e_{i}\lrcorner dV -\oint_{\Sigma} \langle \Tilde{\slashed{D}}^{\Sigma} \psi,
\psi \rangle \, d\Sigma + \frac{1}{2}\oint_{\Sigma} H |\psi|^{2} \, d\Sigma \notag \\
&\rightarrow cm(J,\theta)  -\oint_{\Sigma} \langle \Tilde{\slashed{D}}^{\Sigma} \psi,
\psi \rangle \, d\Sigma + \frac{1}{2}\oint_{\Sigma} H |\psi|^{2} \, d\Sigma \notag
    \end{align}
Also, since we assume $\dim M=5$, $\psi \in \Gamma(\slashed{S}^{odd})$ implies $\Lambda
\nabla _{T}\psi =0$ by (\ref{2-7-1}) as $\nabla _{T}$ leaves $\slashed{S}^{odd}$
invariant. Hence we obtain (\ref{pmassformula}) via (\ref{5-2}).
\end{proof}

\begin{proof}
\textbf{(of Corollary \ref{pmasscor}) }By the assumtion that $\Sigma $ is $p$%
-minimal, we have the $p$-mean curvature $H=0.$ Now (\ref{massineqn}) follows from (\ref{pmassformula}) with the
assumption $W\geq 0$, $H=0$ and the assumption $\slashed{D} \psi =0$ .
\end{proof}

\section{Solving the CR Dirac equation on a domain with boundary\label{Sec6}}

This part is motivated by similar ideas in \cite{farinelli1998spectrum} for
the Riemannian situation. For the CR/subriemannian case, however, we need an
assumption to prove Theorem \ref{T-1-4}.

\begin{myassumpt}
\label{assum} Assume that $\Sigma $ has no singular point. and there is a
transversal $S^{1}-$action $e^{i\vartheta }$ on $\bar{\Omega}$ such that

\begin{itemize}
\item $\frac{\partial}{\partial \vartheta}(e^{i\vartheta} \circ p )=a(p)T
+c^{i}(p)e_{i} =: V$, where $a\neq 0$ for any $p$ in $\bar{\Omega}$.

\item The $S^{1}-$action preserves the boundary and the interior of $\Omega$.

\item The $S^{1}-$action preserves the CR structure.
\end{itemize}
\end{myassumpt}

We note that the standard Clifford (solid) torus in the 3-sphere is an
example.

By assumption \ref{assum}, we can consider the m-th Fourier spinor space 
\begin{equation}
\Gamma _{m}(\overline{\Omega} ,\slashed{S}):=\{\psi \in C^{\infty}(\overline{%
\Omega} ,\slashed{S}):\nabla _{V}\psi =im\psi \}
\end{equation}%
for $m\in \mathbb{Z}$ and%
\begin{equation}
\Gamma _{\leq N}(\overline{\Omega} ,\slashed{S}):=\bigoplus\limits_{|m|\leq
N}\Gamma _{m}(\overline{\Omega} ,\slashed{S}).  \label{6-7}
\end{equation}
Similarly, we define $\Gamma_{\leq N}(\Sigma,\slashed{S})$ on $\Sigma$.

Let $F:=(\sqrt{-1})^{\frac{n}{2}}c(e^{1})...c(e^{2n})$ be the chirality
operator. Then we have 
\begin{equation}
\slashed{D}F +F\slashed{D}=0
\end{equation}
Moreover, using the formula (\ref{moddiraceqnhyper}) for the modified
hypersurface Dirac operator, we have 
\begin{equation}  \label{commuteF}
\Tilde{\slashed{D}}^{\Sigma} F- F\Tilde{\slashed{D}}^{\Sigma}=0
\end{equation}

\begin{mylem}
\label{anticommute} It holds that\newline

$i)$ $\nabla _{X}^{\Sigma }(c(e^{2n})\psi )=c(e^{2n})\nabla _{X}^{\Sigma
}\psi $, \newline

$ii)$ $\Tilde{\slashed{D}}^{\Sigma} (c(e^{2n}) \psi)=-c(e^{2n}) \Tilde{\slashed{D}}%
^{\Sigma} \psi$.
\end{mylem}

\begin{proof}
For the first equality, we have
\begin{align*}
\nabla^{\Sigma}_{X}(c(e^{2n})\psi) =& (\nabla_{X} - \frac{1}{2}
c(AX)c(e^{2n})) (c(e^{2n})\psi) \\
=& (-c(AX) + c(e^{2n})\nabla_{X} +\frac{1}{2} c(AX))\psi \\
=& c(e^{2n})\nabla^{\Sigma}_{X} \psi.
\end{align*}
For the second equality, we have
\begin{align*}
\Tilde{\slashed{D}}^{\Sigma} (c(e^{2n}) \psi) =& (c(e^{i^{\prime
}})c(e^{2n})\nabla^{\Sigma}_{e_{i^{\prime }}} +\alpha
c(e^{2n-1})c(e^{2n}))(c(e^{2n})\psi) \\
=& -c(e^{2n}) \Tilde{\slashed{D}}^{\Sigma} \psi.
\end{align*}
\end{proof}

\begin{mylem}
\label{hyperellipticest} Under the Assumption \ref{assum} and $n\geq 2$, the
modified hypersurface Dirac operator $\Tilde{\slashed{D}}^{\Sigma }$ is a
subelliptic operator on $\Gamma _{\leq N}(\Sigma,\slashed{S})$.
\end{mylem}

\begin{proof}
By integrating the divergence of the one form $A= \langle c^{\Sigma}(e^{i^{\prime }})\Tilde{\slashed{D}}%
^{\Sigma}\psi, \psi\rangle e^{i^{\prime }}$, we have 
\begin{equation}
\int\limits_{\Sigma} \langle (\Tilde{\slashed{D}}^{\Sigma})^{2} \psi, \psi
\rangle d\Sigma=\int\limits_{\Sigma} |\Tilde{\slashed{D}}^{\Sigma}\psi|^{2}
d\Sigma.
\end{equation}
And the Weitzenb\"{o}ck-type formula for the modified hypersurface Dirac operator is of the type
\begin{equation}\label{hyperW}
    (\hyperdirac)^{2}=-\hyperconn_{e_{\ip}}\hyperconn_{e_{\ip}}-2\sum\limits_{\ip}\hyperc(e^{\ip})\hyperc(e^{\ip +n})\hyperconn_{e_{0}}+A^{\ip}\hyperconn_{e_{\ip}}+B,
\end{equation}
where $A^{i}$ and $B$ are the endomorphism from $\slashed{S}^{\Sigma}$ to $\slashed{S}^{\Sigma}$. The derivation of the equation (\ref{hyperW}) is similar to that of  Lemma \ref{Weitzen}. Since any spinor $\psi$ on the space $\Gamma _{\leq N}(\Omega ,\slashed{S})$ has the property $|\nabla_{T}\psi|_{L^{2}}\leq C(N) |\psi|_{L^{2}}$, which can control the $\hyperconn_{e_{0}}$ part. Using this property and the standard argument for elliptic operator, we have an subelliptic estimate.
\end{proof}

By Lemma \ref{selfadjforhyperdirac}, the modified hypersurface Dirac
operator $\Tilde{\slashed{D}}^{\Sigma }$ is self-adjoint in $L^{2}(\Sigma
,\slashed{S})$. Together with the subellipticity by Lemma \ref{hyperellipticest}%
, $\Tilde{\slashed{D}}^{\Sigma }$ has the discrete spectrum as $\Sigma $ is
compact with no boundary. Let $(\psi _{k})$ be the spectral resolution of $%
\Tilde{\slashed{D}}^{\Sigma }$, i.e. $\Tilde{\slashed{D}}^{\Sigma }\psi
_{k}=\lambda _{k}\psi _{k}$. By $ii)$ of Lemma \ref{anticommute}, $%
c(e^{2n})\psi _{k}$ is an eigenspinor with the eigenvalue $-\lambda _{k}$.
So we know $\Tilde{\slashed{D}}^{\Sigma }$ has the spectrum 
\begin{equation}
...\leq \lambda _{-k}\leq ...\leq \lambda _{-1}<0<\lambda _{1}\leq ...\leq
\lambda _{k}\leq ...
\end{equation}%
where $\lambda _{-k}=-\lambda _{k}$ for all $k\in \mathbb{N}$.

Suppose there is the zero eigenvalue. Then we let $\Tilde{\slashed{D}}%
_{\varepsilon }^{\Sigma }:=\Tilde{\slashed{D}}^{\Sigma }+\varepsilon F$. Since
the operator $\Tilde{\slashed{D}}^{\Sigma }$ commutes with $F$ by equation (\ref%
{commuteF}), so there exists an arbitrarily small $\varepsilon >0$ such that 
$\ker (\Tilde{\slashed{D}}_{\varepsilon }^{\Sigma })=\{0\}$ as the spectrum of $%
\Tilde{\slashed{D}}^{\Sigma }$ is discrete. So we can assume there is no zero
eigenvalue of $\Tilde{\slashed{D}}^{\Sigma }$ (by a small perturbation).

\begin{mydef}
\label{D-6-z} Define the following APS (Atiyah-Patodi-Singer) spaces 
\begin{equation}
\Gamma _{+}^{APS}:=\{\psi \in C^{\infty }(\Sigma ,\slashed{S}|_{\Sigma }):\psi
=\sum\limits_{\lambda _{k}>0}c_{k}\psi _{k}\}
\end{equation}%
and 
\begin{equation}
\Gamma _{-}^{APS}:=\{\psi \in C^{\infty }(\Sigma ,\slashed{S}|_{\Sigma }):\psi
=\sum\limits_{\lambda _{k}<0}c_{k}\psi _{k}\}
\end{equation}

Let $\pi _{\pm }:C^{\infty }(\Sigma,\slashed{S})\rightarrow \Gamma _{\pm
}^{APS}$ denote the map obtained by projecting it onto $\Gamma _{\pm }^{APS}$ in $L^{2}$ .
\end{mydef}

Moreover, using the m-th Fourier spinor space, we define the APS condition
on $\Gamma _{\leq N}(\Sigma,\slashed{S})$ by 
\begin{equation}
\Gamma _{+,\leq N}^{APS}:=\{\psi \in \Gamma _{\leq N}(\Sigma
,\slashed{S}):\pi _{+}(\psi )=\psi \}
\end{equation}%
and 
\begin{equation}
\Gamma _{-,\leq N}^{APS}:=\{\psi \in \Gamma _{\leq N}(\Sigma
,\slashed{S}):\pi _{-}(\psi )=\psi \}
\end{equation}

\begin{myrmk}\label{notationrmk}
    For a section $\psi \in C^{\infty}(\overline{\Omega}, \slashed{S})$, the notation $\psi \in \Gamma_{\pm, \leq N}^{APS}$ means $\psi|_{\Sigma} \in \Gamma_{\pm, \leq N}^{APS}$ and the space $C^{\infty}(\overline{\Omega},\slashed{S}) \cap \Gamma_{\pm, \leq N}^{APS}$ means $\{\psi \in C^{\infty}(\overline{\Omega},\slashed{S}) :  \psi|_{\Sigma} \in \Gamma_{\pm, \leq N}^{APS}\}$.
\end{myrmk}

\begin{mypro}
It holds that (for the notation $C^{\infty}(\overline{\Omega},\slashed{S}) \cap \Gamma^{APS}_{\pm, \leq N}$, see Remark \ref{notationrmk}.)
\begin{equation*}
\slashed{D}: C^{\infty}(\overline{\Omega},\slashed{S}) \cap \Gamma^{APS}_{\pm,
\leq N} \rightarrow
C^{\infty}(\overline{\Omega},\slashed{S}) \cap \Gamma^{APS}_{\pm, \leq N}.
\end{equation*}

\end{mypro}

\begin{proof}
As the action preserves the CR structure, we have $[\nabla_{V}, c(e^{i})]=0$
and $[\nabla_{V},\nabla_{e_{i}}]=0$. Hence the proposition follows.
\end{proof}

\begin{mylem}
\label{ellipticest} Under the local boundary condition $\psi \in
C^{\infty}(\overline{\Omega},\slashed{S}) \cap\Gamma^{APS}_{-,\leq N}$, the Dirac operator $\slashed{D}$ has a subelliptic
estimate, i.e. for any $\delta>0$ and $N$ large enough, there exists $%
C_{\delta}=C_{\delta}(N,n)>0$ such that 
\begin{equation}  \label{ellipticformula}
||\psi||^{2}_{H^{1}_{FS}} \leq (1+\delta) ||\slashed{D} \psi||^{2}_{L^{2}} +
C_{\delta} ||\psi||^{2}_{L^{2}}
\end{equation}
for all $\psi \in C^{\infty}(\overline{\Omega},\slashed{S}) \cap\Gamma^{APS}_{+,\leq N}$
, where $\|\cdot\|_{H^{1}_{FS}}$ denotes the first-order Folland--Stein norm.
\end{mylem}

\begin{proof}
By Assumption \ref{assum}, we have the following estimate:
\begin{align}
    &\int\limits_{\Omega} \re\langle \Lambda \nabla_{T} \psi,\psi \rangle dV \label{6-11a}\\
    =&\int\limits_{\Omega} \re[ \big\langle \Lambda \nabla_{a^{-1}(V-c^{i}e_{i})} \psi, \psi \big\rangle] dV \notag\\
    \leq& a^{-1}N||\psi||^{2}_{L^{2}(\Omega)} + |c|(\varepsilon_{0} ||\psi||^{2}_{H^{1}_{FS}(\Omega)}+C_{\varepsilon_{0}}||\psi||^{2}_{L^{2}(\Omega)}) \notag\\
    \leq & \varepsilon_{1} ||\psi||^{2}_{H^{1}_{FS}(\Omega)}+(a^{-1}N+C_{\varepsilon_{1}})||\psi||^{2}_{L^{2}(\Omega)} \notag
\end{align}

Via the CR Reilly formula (\ref{reillyeqn}), we have

\begin{align*}
&||\psi||_{H^{1}_{FS}}^{2}=\int_{\Omega} |\nabla \psi|^{2} dV \\
=& \oint_{\Sigma} \langle \Tilde{\slashed{D}}^{\Sigma} \psi, \psi \rangle \,
d\Sigma - \frac{1}{2}\oint_{\Sigma} H |\psi|^{2} \, d\Sigma \\
&+ \int_{\Omega} \Big[ |\slashed{D} \psi|^{2} - W |\psi|^{2} + 2\re[
\big\langle \Lambda \nabla_{T} \psi, \psi \big\rangle] \Big] dV \\
\leq & -\frac{1}{2}\oint_{\Sigma} H |\psi|^{2} \, d\Sigma + \int_{\Omega} %
\Big[ |\slashed{D} \psi|^{2} - W |\psi|^{2} + 2\re[ \big\langle \Lambda
\nabla_{T} \psi, \psi \big\rangle] \Big] dV \\
\leq & \varepsilon_{2} ||\psi||_{H^{1}_{FS}}^{2} + ||\slashed{D} \psi||_{L^{2}}^{2}+
C_{1}||\psi||_{L^{2}}^{2} + \int_{\Omega} 2\re[ \big\langle \Lambda
\nabla_{a^{-1}(V-c^{i}e_{i})} \psi, \psi \big\rangle] dV \\
\leq& (\varepsilon_{2}+2\varepsilon_{1})||\psi||_{H^{1}_{FS}}^{2} + ||\slashed{D} \psi||^{2}_{L^{2}}+
(2C_{\varepsilon_{1}}+C_{1}+2a^{-1}N)||\psi||^{2}_{L^{2}}   \\
\leq& \epsilon||\psi||_{H^{1}_{FS}}^{2}+ ||\slashed{D} \psi||^{2}_{L^{2}}
+(2a^{-1}N+C)||\psi||_{L^{2}}
\end{align*}

For the inequality $\oint H|\psi|^{2} d\Sigma \leq \varepsilon_{2}
||\psi||_{H^{1}_{FS}}^{2}$, we use the pseudohermitian trace theorem, see \cite%
{bahouri2009trace}. We have also used (\ref{6-11a}) to estimate the term involving $\big\langle \Lambda \nabla_{T} \psi, \psi \big\rangle $. Taking the suitable $\epsilon$, we get the inequality (\ref{ellipticformula}).
\end{proof}



We denote by $L^{2}\Gamma _{+}^{APS}$ (resp. $L^{2}\Gamma _{\pm ,\leq
N}^{APS})$ the $L^{2}$-completion of $\Gamma _{+}^{APS}$ (resp. $\Gamma
_{\pm ,\leq N}^{APS})$. Similar notations applyto other spaces.

\begin{mythm}
Fixed $N>0$. Let $(\Omega ,\Sigma =\partial \Omega ,L_{\theta
},c,\slashed{S},\nabla )$ be the spinor bundle with chirality $F$. Then the CR
Dirac operator $\Tilde{\slashed{D}}^{\Sigma }$ extends to a self-adjoint linear
operator on $L^{2}(\Omega ,\slashed{S})$ with domain 
\begin{equation}
D_{-,N}(\slashed{D}):=\{\psi \in H^{1}_{FS}(\Omega ,\slashed{S})\cap L^{2}\Gamma _{\leq
N}(\overline{\Omega },\slashed{S}):\psi|_{\Sigma} \in L^{2}\Gamma _{-,\leq N}^{APS}\}.
\end{equation}
\end{mythm}

\begin{proof}
By integrating the divergence of the one form $\gamma(X)=\langle c(X)\phi,\psi \rangle$, we
have 
\begin{equation}  \label{selfadjD}
\llangle \slashed{D}\phi, \psi \rrangle-\llangle\phi, \slashed{D}\psi\rrangle =
\oint\limits_{\Sigma} \langle c(e^{2n})\phi, \psi \rangle
\end{equation}

By Lemma \ref{anticommute}, if $\phi, \psi \in D_{+,N}(\slashed{D})$, then $%
c(e^{2n})\phi \in L^{2}\Gamma^{APS}_{-}$. Thus the boundary term of %
(\ref{selfadjD}) vanishes.

Let $\slashed{D}^{*}$ be the adjoint operator. Its domain is 
\begin{equation}  \label{adjdomain}
\begin{aligned} D_{-,N}(\slashed{D}^{*}):=\{ &\zeta \in
L^{2}\Gamma_{\leq N} (\Omega,\slashed{S}) : \exists \chi \in
L^{2}\Gamma_{\leq N}(\Omega,\slashed{S})  \\ &\text{ with } \llangle
\chi,\psi\rrangle=\llangle\zeta, \slashed{D} \psi\rrangle \forall \psi\in
D_{-,N}(\slashed{D}) \} \end{aligned}
\end{equation}

By (\ref{selfadjD}), we have 
\begin{equation}
\chi=\slashed{D}\zeta,\text{ and } \oint\limits_{\Sigma} \langle c(e^{2n})
\zeta, \psi \rangle=0 \text{ } \forall \psi \in D_{-,N}(\slashed{D}).
\end{equation}

\noindent Thus $\zeta \in \Gamma^{APS}_{-,\leq N}$. For general $\zeta$, using the
subelliptic estimate (Lemma \ref{ellipticest}) similar to the one in \cite%
{farinelli1998spectrum}, we conclude the result.
\end{proof}

\begin{mydef}
\label{D-6-1} For $\lambda \in \mathbb{R}$, define $N_{\lambda
}(\slashed{D}):=\{\psi \in D_{-,N}(\slashed{D}):(\slashed{D}-\lambda )\psi =0\}$
\end{mydef}

The theorem below follows from a similar reasoning in \cite%
{farinelli1998spectrum}.

\begin{mythm}
\label{spectrumthm} It holds that\newline
1. The space $N_{0}(\slashed{D})$ is finite dimensional.\newline
2. $\slashed{D}$ with domain $D_{+,N}(\slashed{D})$ is a Fredholm operator.\newline
3. $L^{2}\Gamma_{\leq N}(\Omega ,\slashed{S})=N_{\mu }(\slashed{D})\oplus \im%
(\slashed{D}-\mu )$.
\end{mythm}

For the general inhomogeneous boundary value problem (\ref{1-5}) for the
Dirac operator, we have the following result.

\begin{proof} \textbf{(of Theorem \ref{T-1-4}) }
Consider $\phi -\slashed{D}\rho \in L^{2}_{\leq N}(\Omega,\slashed{S})$. By
Theorem \ref{spectrumthm}, $\phi -\slashed{D}\rho \in \im(\slashed{D})$ iff $\phi
-\slashed{D}\rho \in N_{0}(\slashed{D})^{\perp}$, proving the theorem.
\end{proof}









\appendix

\section{Some useful formulas}

Replacing $\psi $ by $\slashed{D}\psi $ in (\ref{reillyeqn}), and integrating
by parts, we have%
\begin{equation}
\begin{aligned} &-\oint_{\Sigma} \langle \Tilde{\slashed{D}}^{\Sigma}
\slashed{D}\psi, \slashed{D}\psi \rangle \, d\Sigma + \frac{1}{2}\oint_{\Sigma} H
|\slashed{D}\psi|^{2} \, d\Sigma \\ &-\oint\limits_{\Sigma} \langle
c(e^{2n})\Lambda\nabla_{T}\slashed{D} \psi, \psi\rangle-\oint\limits_{\Sigma}
\langle \psi, c(e^{2n})\Lambda\nabla_{T}\slashed{D} \psi\rangle\\ &=
\int_{\Omega} \Big[ -|\slashed{D}^{2} \psi|^{2} + W |\slashed{D}\psi|^{2} +
|\nabla \slashed{D}\psi|^{2} - \big\langle \slashed{D}\Lambda \nabla_{T}
\slashed{D}\psi, \psi \big\rangle -\big\langle \psi, \slashed{D}\Lambda \nabla_{T}
\slashed{D}\psi \big\rangle \Big] dV. \end{aligned}  \label{dreillyeqn}
\end{equation}

To study the self-adjointness of the operator $P:=\slashed{D}\Lambda \nabla
_{T}\slashed{D}$, we have the following useful formula:%
\begin{eqnarray}
\llangle P\phi ,\psi \rrangle-\llangle\phi ,P\psi \rrangle
&=&\oint\limits_{\Sigma }\langle c(e^{2n})\Lambda \nabla _{T}\slashed{D}\phi
,\psi \rangle   \label{selfadjP} \\
&&-\oint\limits_{\Sigma }\langle \phi ,c(e^{2n})\Lambda \nabla
_{T}\slashed{D}\psi \rangle +\oint\limits_{\Sigma }\alpha \langle \slashed{D}\phi
,\Lambda \slashed{D}\psi \rangle .  \notag
\end{eqnarray}

To study the self-adjointness of the operator $\slashed{D}^{2}$, we have the
following useful formula:%
\begin{equation}
\llangle \slashed{D}^{2}\phi ,\psi \rrangle-\llangle\phi ,\slashed{D}^{2}\psi %
\rrangle=\oint\limits_{\Sigma }\langle c(e^{2n})\slashed{D}\phi ,\psi \rangle
-\oint\limits_{\Sigma }\langle \phi ,c(e^{2n})\slashed{D}\psi \rangle
\label{selfadjD2}
\end{equation}

\noindent When the torsion is free, we have 
\begin{equation}
\lbrack \slashed{D},\nabla _{T}]=0.
\end{equation}

\medskip

\bigskip


\end{document}